\documentclass[graybox]{svmult}
\usepackage{xcolor}

\usepackage{type1cm}        
\usepackage{makeidx}         
\usepackage{graphicx}        
\usepackage{multicol}        
\usepackage[bottom]{footmisc}

\usepackage{newtxtext}       %
\usepackage{newtxmath}       

\makeindex             

\usepackage{mathtools}
\usepackage{mathdesign}
\mathtoolsset{showonlyrefs} 
\usepackage{xcolor}

\newcommand{\R}{{\mathbb{R}}}

\newcommand{\al}{\alpha}

\newcommand{\ga}{\gamma}

\newcommand{\om}{\omega}

\newcommand{\beq}{\begin{equation}}
\newcommand{\eeq}{\end{equation}}
\newcommand{\bdm}{\begin{displaymath}}
\newcommand{\edm}{\end{displaymath}}
\newcommand{\ba}{\begin{align}}
\newcommand{\ea}{\end{align}}
\newcommand{\bpf}{\begin{proof}}
\newcommand{\epf}{\end{proof}}
\newcommand{\eps}{\epsilon}

\allowdisplaybreaks
\newcommand{\Id}{\operatorname{Id}}

\newcommand{\cB}{\mathcal{B}}

\newcommand{\cG}{\mathcal{G}}

\newcommand{\del}{\partial}
\newcommand{\half}{\frac{1}{2}}

\begin{document}

\title*{A Note on Singularity Formation for a Nonlocal Transport Equation}
\author{Vu Hoang and Maria Radosz}
\institute{Vu Hoang \at Department of Mathematics, University of Texas at San Antonio, One UTSA Circle, TX 28249, \email{duynguyenvu.hoang@utsa.edu}
\and Maria Radosz \at Department of Mathematics, University of Texas at San Antonio, One UTSA Circle, TX 28249, \email{maria\_radosz@hotmail.com}}
%
%
\maketitle

\abstract*{The $\al$-patch model is used to study aspects of fluid equations. We show that solutions of this model form singularities in finite time and give a characterization of the solution profile at the singular time.}

\abstract{The $\al$-patch model is used to study aspects of fluid equations. We show that solutions of this model form singularities in finite time and give a characterization of the solution profile at the singular time.}

%
\section{Introduction}
%

One of the most fundamental equations in modeling the motion of fluids and gases is the transport equation
\begin{equation}\label{eq1}
\omega_t+u\cdot \nabla \omega=0,
\end{equation}
here written using the vorticity $\om$.  The velocity $u$ may depend on $\omega$, in which case \eqref{eq1} is called an active scalar equation.

When the relationship $u[\om]$ is specified, \eqref{eq1} gives rise to many important models in fluid dynamics. $u[\om]$ is often called a Biot-Savart law. Here are some examples of particular transport equations, which are also active scalar equations. 
Take for example,
\begin{equation}\label{eq2}
u=\nabla^{\perp}(-\triangle)^{-1}\omega,
\end{equation}
where $\nabla^{\perp}=(-\partial_y,\partial_x)$ is the perpendicular gradient. Equations \eqref{eq1} and \eqref{eq2} are the vorticity form of 2D Euler equations.
As another example, one can take
$$
u=\nabla^{\perp}(-\Delta)^{-\frac{1}{2}}\omega.
$$
Then \eqref{eq1} becomes the surface quasi-geostrophic (SQG) equation. The SQG equation has important applications in geophysics and atmospheric sciences \cite{Ped}. Moreover it serves as a toy model for the 3D-Euler equations (see \cite{constantin1994formation} for more details).

A question of great importance is whether solutions for these equations form singularities in finite time. 

A game-changing observation in dealing with some two- and three dimensional models for fluid motion is that imposing certain symmetries on the solution of \eqref{eq1} on simple domains like a half-disc or a quadrant creates a special kind of flow called a \emph{hyperbolic flow}. We will not go into details here but refer to \cite{Hoang2D,HouLuo2,HouLuo3,kiselev2013small} for more information. In this hyperbolic flow scenario, it seems that the behavior of the fluid on and near the domain boundary plays the most important part in creating either blowup or strong gradient growth. One-dimensional models capturing this behavior are therefore an essential tool for investigating possible blowup mechanisms without the additional complications of more-dimensional equations. We refer to \cite{sixAuthors} for discussion of the aspects relating to the hyperbolic flow scenario.  For 1d models of fluid equations, see also \cite{sixAuthors, CKY, CLM, CCF}. 

In this paper, we will study a 1D model of (\ref{eq1}) on $\R$ with the following Biot-Savart law:
\begin{eqnarray}
\label{eq4} u&=&(-\Delta)^{-\frac{\alpha}{2}}\omega=-c_\alpha\int_{\R}|y-x|^{-(1-\alpha)}\omega(t,y)~ dy.
\end{eqnarray}

This model is called $\alpha$-patch model and has also been treated in \cite{dong2014one} with an additional viscosity term causing dissipation. Local existence of the solution and the existence of a blowup for the viscous $\alpha$-patch model were given in \cite{dong2014one}. The existence of blowup was obtained by using energy methods. In contrast, this paper deals with more geometric aspects of singularity formation for the inviscid model - such as the final profile of the solution at the singular time.

Regularity-wise, the $\alpha$-patch model is less regular than 1D Euler $u_x = H \om$ and more regular than the C\'ordoba-C\'ordoba-Fontelos (CCF) model $u=H\om$ (see \cite{CCF}). The latter is a 1D analogue of the SQG equation.

%
\section{One dimensional $\alpha$-patch model and main results}
%

We study the transport equation
\beq\label{eq_main}
\omega_t+u[\om]\omega_x=0
\eeq
in one space dimension for the unknown function $\om(t, x) : [0, T]\times \R \to \R$ with sufficiently smooth initial data $\om(\cdot,0)=\om_0$. The velocity field is given by the nonlocal Biot-Savart law
\beq
u(x,t)= (-\Delta)^{-\alpha/2}\omega(x,t) =-c_\alpha\int_{\R}|y-x|^{-(1-\alpha)}\omega(t,y)~dy,\qquad\alpha\in(0,1).
\eeq
 The $\al$-patch model becomes the 1d model for the 2d Euler equation in the limit $\al\to 1$ with velocity field given by
\begin{align}
    u(x,t)= (-\Delta)^{-1/2}\omega(x,t).
\end{align}

For convenience, we will assume the constant $c_\alpha$ associated with the fractional Laplacian is $1$, and we write $\ga = 1-\al$.

We consider classical solutions where $\om(\cdot, x)$ is odd in $x$ and such that
\begin{align}
    \|\om(t, \cdot)\|_{q} + \|\om_x(t, \cdot)\|_{q-1} < \infty
\end{align}
for all $t\in [0, T]$. Here, the norm $\|\cdot\|_s$ is defined by
\begin{align}
    \|\om\|_s := \sup_{x \geq 0}|\om(x)|(1+x)^{-s}.
\end{align}
Our main concern will be the question if more information can be deduced about the nature of the singularity formation of \eqref{eq_main}. In particular, we are interested  in  the formation of an \emph{odd cusp} (see Fig.~\ref{cusps}), i.e. the possibility that a smooth solution becomes singular at the time $t=T_s>0$ in a way such that
\begin{align}
    \om \sim \operatorname{sign}(x)|x|^p\quad (t\to T_s)
\end{align}
with some power $p\in (0, 1)$. The sense in which this holds will be made clear below. Another result on cusp formation can be found in \cite{HoangRadoszCusp}.

\begin{figure}[htbp]
\begin{center}
\includegraphics[scale=0.36]{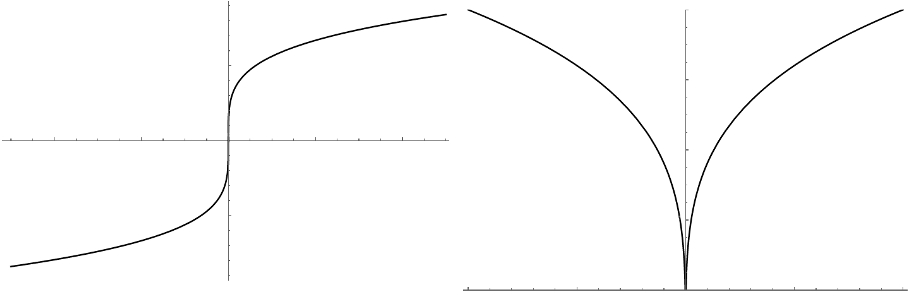}
\caption{Left: odd cusp, Right: (even) cusp. } \label{cusps}
\end{center}
\end{figure}

We shall take odd $C^2$-smooth initial data: 
\begin{align}\label{eq_sym1}
    \om(0, -x) = - \om(0, x)=\om_0(x) \quad (x\in \R)
\end{align}
This implies that $\om(t,\cdot)$ is odd for all $t$ (as long as a smooth solution exists) and also that $u[\om](t, \cdot)$ is odd as well. Moreover, $\om_0$ is such that
\begin{align}\label{eq_init1}
    \|\om_0\|_{q} + \|\del_x\om_0\|_{q-1} + \|\del_x^2 \om_0\|_{q-2} < \infty
\end{align}
for some $0 < q < \ga$.
Note that due to the symmetry, the velocity field is well-defined for $\om$ satisfying \eqref{eq_init1}. This is because
\begin{align}\label{eq_form1}
    u[\om](t, x) = -\int_{0}^\infty K(x, y)\om(t,y)~dy
\end{align}
where 
\begin{align}
    K(x, y) = \left(\frac{1}{|y-x|^\ga}-\frac{1}{|x+y|^\ga}\right) \geq 0
\end{align}
and hence $|K(x, y)| \leq C(x)|y|^{-1-\ga}$ for $y\geq 2x$. This implies that \eqref{eq_form1}  converges for all $x>0$.  Note in particular that
$u[\om](x)\leq 0$ for $x\geq 0$, if $\om(x)\geq 0$ for $x\geq 0$. 

\begin{theorem}[Local Existence and Uniqueness]\label{thm:shorttimeexist}Given $C^2$-initial data $\om_0$ satisfying \eqref{eq_init1}, there exists a $T>0$ and a unique solution $\om$ of \eqref{eq_main} defined on $[0,T]\times\R$  so that 
\begin{itemize}
    \item $\om(0, \cdot)=\om_0(\cdot)$
    \item $\om(t, \cdot)$ is odd for all $t\in [0, T)$.
    \item $\om \in C^1([0, T]\times\R)$
\end{itemize}
\end{theorem}

Define now
\begin{align}\label{def_phi}
    \phi(t, x) = a^p(t) f\left(\frac{x}{a(t)}\right)
\end{align}
with $f(z) = (z+1)^p - 1$.
The function $\phi$ serves a barrier for solutions of \eqref{eq_main}, as shown by our main result below. The function $a(t)$ controls the evolution of the barrier's shape in time and will also be specified in Theorem \ref{thm_MR}. Suppose that $a(t)\to 0$, as $t\to T^*$. Then as $t\to T^*$
\begin{align}
    \phi(t, x) \to x^p
\end{align}
pointwise for $x>0$  and also uniformly on $x\geq c$ for any $c > 0$. Note moreover that 
\begin{align}\label{eq_phi_less_than_x_plus_a_0}
    \phi(t, x) \leq (x+a_0)^p \quad (t \geq 0, x\geq 0)
\end{align}
provided $0 < a(t) \leq a_0$.

\begin{theorem}[Singularity formation]\label{thm_MR}
Let $p=\half \ga$. There exists a $c_0>0$ such that the following implication is true:
If $a: [0, T(a_0)]\to \R$ solves
\begin{align}
    \dot a = - c_0 a^{1-p}, ~~a(0) = a_0 >0
\end{align}
with $a_0 < 1$ and $T(a_0) > 0$ being the unique time such that
\begin{align}
    a(T(a_0)) = 0, ~~a(t) > 0~\text{for}~t < T(a_0).
\end{align}
and if $\om$ is a smooth solution of \eqref{eq_main} such that
\begin{align}\label{assumption}
\begin{split}
    &\|\om(0,\cdot)\|_{p} + \|\om_x(0,\cdot)\|_{p-1} < \infty, \\
    &\om(0, x) > (1+\eps)\phi(0, x)\quad (x>0)
\end{split}
\end{align}

for $\eps>0$ satisfying
\begin{align}\label{choice_eps}
    \eps > (1-a_0^p)^{-1} - 1.
\end{align}
Then $\bar T \leq T(a_0)$ and
\begin{align}
    \om(t, x) > \phi(t, x)\quad (x>0)
\end{align}
for $0\leq t < \bar T$, where $\bar T>0$ denotes the maximal lifetime of the smooth solution.  Provided $\om$ does not break down earlier, $\om$ forms at least a cusp at time $t = T(a_0)$ (or a potentially stronger singularity, see Fig. \ref{singularities}).
\end{theorem}
Note that a function $a(t)$ with the properties in the Theorem exists for any given $a(0) > 0$. 
 Our theorem does not exclude the possibility of a singularity forming before $T(a_0)$. However, we offer the following conjecture.
\begin{conjecture}
If $\om_{xx}(0, x) < 0$ for $x > 0$, then singularity formation can only happen at $x = 0$, i.e. if $\sup_{0\leq t < \hat{T}}|\om_x(t, 0)| < \infty$ for some $ \hat{T} > 0$, then the smooth solution can be continued past $t=\hat{T}$. In this case we have that the profile of $\om(t, \cdot)$ converges to an odd cusp at the singularity. 
\end{conjecture}

\begin{figure}[htbp]
\begin{center}
\includegraphics[scale=0.36]{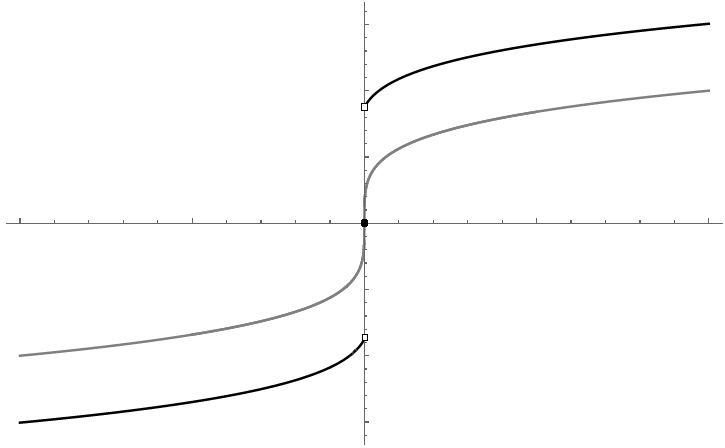}
\caption{Possible singularity fomations. Grey: Odd cusp, Black: Shock. } \label{singularities}
\end{center}
\end{figure}

%
\section{Proofs of main results}
%

\begin{proposition}\label{prop_est_u}
For all $0 < q< \ga$, we have the estimates
\begin{align}\label{est_u}
\begin{split}
    \|u[\om]\|_{1-\ga+q}&\leq C\|\om\|_{q}\\
    \|\del_{x}u[\om]\|_{-\ga+q}&\leq C\|\om_x\|_{q-1}\\
    \|\del_{xx}u[\om]\|_{-\ga+q-1}&\leq C\|\om_{xx}\|_{q-2}.
\end{split}
\end{align}
with a universal constant $C>0$.
\end{proposition}

\begin{proof}
We estimate
\begin{align}
    |u[\om]|&\leq \|\om\|_{q}\int_{0}^\infty K(x, y) (1+y)^{q}dy\\
    &=\|\om\|_{q}\int_{0}^\infty x^{-\ga} K\left(1, \frac{y}{x}\right) x^{q} \left(x^{-1} +\frac{y}{x}\right)^{q}dy\\
    &=\|\om\|_{q}x^{1-\ga+q} \int_{0}^\infty K\left(1, z\right) (x^{-1}+z)^{q}dz
\end{align}
after making the substitution $y = x z$. Now note that for $x\geq 1$, $(x^{-1}+z)^{q}\leq (1+z^q)$, hence with
\begin{align}
    C := \int_{0}^\infty K\left(1, z\right) (1+z)^{q}dz 
\end{align}
the estimate $\sup_{x\geq 1} |u[\om]| x^{-(1-\ga+q)} \leq C \|\om\|_q$ holds. Note that $C < \infty$ on account of $0 < q <\ga$.  Straightforward computations show that $|u[\om](x)|\leq C \|\om\|_q$ for $0\leq x \leq 1$, hence the first line of \eqref{est_u} holds.
To continue with the second line of \eqref{est_u}, we first note 
\begin{align}
    u[\om]_x = -\int_{\R} \frac{\om_x(y)}{|x-y|^{\ga}}~dy = -\int_{0}^\infty \om_x(y)\left(|x-y|^{-\ga}+|x+y|^{-\ga}\right)~dy
\end{align}
where $\om_x$ is an even function and the integral is absolutely convergent. Similar estimations now show the second line and third line of \eqref{est_u}.
\end{proof}


\subsection{Proof of Theorem \ref{thm:shorttimeexist} (Local Existence and Uniqueness)}

The proof consists of two parts. We first construct global solutions of 
of an approximate problem. The following a-priori bounds are crucial:

\begin{proposition}\label{prop4}
Suppose $0 < p < \ga$, $\om_0 \in C^2(\R)$ is odd and
\begin{align}
    \|\om_0\|_{p}+\|\del_x \om_0\|_{p-1}+\|\del_{xx}\om_0\|_{p-2} < \infty.
\end{align}
Let $\om:[0, \infty)\times \R\to \R$ solve the equation $\om_t + v \om_x = 0$ where the velocity field satisfies
\begin{align}
    \|v(t, \cdot)\|_{1-\ga+p} \leq K\|\om(t, \cdot)\|_p, \\
    \|\del_x v(t, \cdot)\|_{-\ga+p} \leq K \|\om_x(t, \cdot)\|_{p-1}\\
    \|\del_{xx} v(t, \cdot)\|_{-\ga+p-1} \leq K \|\om_{xx}(t, \cdot)\|_{p-2}
\end{align}
for all $t \geq 0$ and some $K > 0$. Then there exists a time $T^*>0$ and a $C>0$ depending only on $\om_0$ and $K$ such that
\begin{align}\label{eq_statement}
    \sup_{0\leq t\leq T^*}\{ \|\om(t, \cdot)\|_{p}+\|\del_x \om(t, \cdot)\|_{p-1}
    + \|\del_{xx} \om(t, \cdot)\|_{p-2}\}\leq C < \infty
\end{align}
holds.
\end{proposition}

\begin{proof}
Along any particle trajectory $X(t)$ we compute 
\begin{align}
  &  \frac{d}{dt} \left((1+X(t))^{-p}\om(t, X(t))\right)\\
   &\qquad\qquad= -p (1+X(t))^{-p-1} v(t, X(t))\om(t, X(t))\\
    &\qquad\qquad\qquad+ (1+X(t))^{-p} \frac{d}{dt} \om(t, X(t))\\
    &\qquad\qquad=-p (1+X(t))^{-p-1} v(t, X(t))\om(t, X(t))\\
    &\qquad\qquad\leq (1+X(t))^{-p-1} K\|\om(t, \cdot)\|_p (1+X(t))^{1-\ga+p} |\om(t, X(t))|\\
     &\qquad\qquad=  K\|\om(t, \cdot)\|_p (1+X(t))^{-\ga+p} (1+X(t))^{-p} |\om(t, X(t))|\\
     &\qquad\qquad\leq  K \|\om(t, \cdot)\|_p^2 \sup_{x\geq 0} (1+x)^{-\ga+p}\\
    &\qquad\qquad\leq  K \|\om(t, \cdot)\|_p^2
\end{align}
where we have used $\ga > p$. A similar computation shows that
\begin{align}
    -\frac{d}{dt} \left((1+X(t))^{-p}\om(t, X(t))\right) \geq - K \|\om(t, \cdot)\|_p^2
\end{align}
and hence there exists a $T^*>0$ depending only on $\|\om_0\|_p, K$ such that $\|\om(t, \cdot)\|_p$ is bounded by a constant on $[0, T^*]$. To prove a similar bound for $\del_x \om(t, \cdot)$, we observe that
\begin{align}
    &\frac{d}{dt} \left((1+X(t))^{-p+1}\om_x(t, X(t))\right) \\
    &\quad=
    (-p+1) (1+X(t))^{-p} v(t, X(t)) \om_x(t, X(t)) \\
    &\quad\qquad- (1+X(t))^{-p+1}(\del_x v)(t, X(t)) \om_x(t, X(t)) \\
    &\quad\leq  K\|\om_x(t, \cdot)\|_{p-1} (1+X(t))^{-1}(1+X(t))^{1-\ga+p}\|\om(t, \cdot)\|_p \\
    &\quad\qquad+K (1+X(t))^{-p+1} (1+X(t))^{-\ga+p}\|\om_x(t,\cdot)\|_{p-1} |\om_x(t, X(t))|\\
    &\quad\leq K \|\om(t, \cdot)\|_p \|\om_x(t, \cdot)\|_{p-1}\\
    &\quad\qquad+ K (1+X(t))^{-p+1} (1+X(t))^{-\ga+p}\|\om_x(t,\cdot)\|_{p-1} (1+X(t))^{p-1}\|\om_x(t, X(t))\|_{p-1}\\
    &\quad\leq C K \left(\|\om(t, \cdot)\|_p \|\om_x(t, \cdot)\|_{p-1}+ \|\om_x(t,\cdot)\|_{p-1}^2\right)
\end{align}
with some universal $C>0$. A similar lower bound for $-\frac{d}{dt} (1+X(t))^{-p+1} \om_x$ exists. We hence get an a-priori bound for $\|\del_x \om(t,\cdot)\|_{p-1}$. A similar argument for $\del_{xx}\om$ completes \eqref{eq_statement}.
\end{proof}

We now define a family of regularized problems. Set
\begin{align}
    k_\eps(z) = \eta_\eps^{-\ga}(|z|)
\end{align}
where $\eta_\eps(z) = \eps \eta(z/\eps)$ with $\eta$ being a smooth, nonincreasing function with the properties
\begin{align}
    \eta(z) = \frac{3}{4} \quad (z\in [0, \frac{3}{4}])\\
    \eta(z) = z \quad (z \geq 1).
\end{align}
Now define for odd $\om$
\begin{align}
    v_\eps[\om] = -\int_\R k_\eps(x-y) \om(y,t)~dt
\end{align}
which can also be written as
\begin{align}\label{eq:u}
    v_{\eps}[\om](t, x) &= -\int_{0}^\infty \left(k_\eps(x-y)-k_\eps(x+y)\right)\om(t, y)~dy.
\end{align}
We note the following estimates:
\begin{align}\label{est_v}
\begin{split}
    \|v_\eps[\om]\|_{1-\ga+p}&\leq K\|\om\|_{p}\\
    \|\del_{x}v_\eps[\om]\|_{-\ga+p}&\leq K\|\om_x\|_{p-1}\\
    \|\del_{xx}v_\eps[\om]\|_{-\ga+p-1}&\leq K\|\om_{xx}\|_{p-2}.
\end{split}
\end{align}
with some universal $K > 0$. This is shown similarly as in Proposition \ref{prop_est_u}, noting that the regularized kernel $k_\eps$ is bounded by the original kernel $|z|^{-\ga}$. On the other hand, we have estimates of the form
\begin{align}\label{est_v1}
\begin{split}
    \|\del_{x}v_\eps[\om]\|_{-\ga+p}&\leq C(\eps)\|\om\|_{p}\\
    \|\del_{xx}v_\eps[\om]\|_{-\ga+p-1}&\leq C(\eps)\|\om\|_{p}.
\end{split}
\end{align}
with $C(\eps)\to \infty$ as $\eps\to 0$.

\begin{proposition}
For all $\eps > 0$, the regularized problems
\begin{align}\label{eq_reg}
    \om_t + v_\eps[\om] \om_x = 0, ~~~\om(0, x) = \om_0
\end{align}
have solutions $\om\in C^2([0,\infty)\times \R)$.
\end{proposition}
\begin{proof}
The first step is to show the local-in-time existence of solutions using the \emph{particle trajectory method} (see \cite{MajdaBertozzi}). The flow map $\Phi=\Phi(t,z)$ satisfies the following equation: 
\begin{align*}
\frac{d\Phi}{dt}(t,z)=v_\eps[\om_\Phi](\Phi(t,z),t), \; \Phi(z,0)=z.
\end{align*}
or equivalently
\begin{align}\label{eq:Phi1}
\Phi(t,z)=z+ \int_0^t v_\eps[\om_\Phi(\cdot,s)](\Phi(s,z),s)~ds.
\end{align}
Here, for a given a flow map $\Phi$, we define 
\beq\label{eq:om2}
\om_\Phi(t,y):=\om_0(\Phi^{-1}(t,y)).
\eeq 
This means, \eqref{eq:Phi1} is an equation for $\Phi$ with
velocity field given by \eqref{eq:u} and $\om_\Phi$ given by
\eqref{eq:om2}. Moreover, a solution of \eqref{eq:Phi1} translates into a solution of \eqref{eq_reg} via relation \eqref{eq:om2}. 

We define the operator $\cG$ formally by
\beq
\cG_\eps[\Phi](x,t):= x+ \int_0^t v_{\eps}[\om_\Phi](\Phi(s,x),s)~ ds
\eeq
with $v_{\eps}$ defined by the expression \eqref{eq:u} where $\om_\Phi$ is given by \eqref{eq:om2}.
Then solving \eqref{eq_reg} is equivalent to solving the fixed point equation
\begin{align*}
\cG_\eps[\Phi] = \Phi.
\end{align*}
Next we need to introduce a suitable metric space on which $\cG$ is well defined and a contraction. To ease notation, we now fix $\eps>0$ and henceforth drop the subscript $\eps$.

\begin{definition}
Let $\cB$ be the set of all $\Phi\in C([0,T],C^2([0,\infty)))$
with the following properties:
\begin{align}
&\Phi(t,0) = 0 \label{prop_1} \\
&\Phi(t,\R^+)\subseteq \R^+.\label{prop_2}
\end{align}
Moreover, $\Phi$ is of the form
\begin{align}\label{prop_4}
\Phi=\Id+\widehat\Phi, \quad \|\widehat \Phi\|\le\zeta
\end{align}
where $\Id$ means the mapping $\Id(t,z)=z$ and
\begin{align}
\|\widehat\Phi\|:=\sup_{t\in [0,T]}\left(\|\widehat\Phi(t,\cdot)\|_{1-\ga+p}+\|\widehat\Phi_z(t,\cdot)\|_{-\ga+p} +\|\widehat\Phi_{zz}(t,\cdot)\|_{-1-\ga+p} \right).
\end{align}
$\cB$ is a complete metric space with metric
$$
d(\Id+\widehat \Phi,\Id+\widehat \Psi) = \|\widehat \Phi-\widehat \Psi\|.
$$
\end{definition}

Note that for sufficiently small $\zeta > 0$ and for any $\Phi\in \cB$ and $t\in[0,T]$ 
$$\Phi(t,\cdot):\R^+\to\R^+$$
is a diffeomorphism of $(0, \infty)$ onto $(0,\infty)$.
To show that, we first note that $\Phi(t,(0,\infty))\subset (0,\infty)$  by
\eqref{prop_2}. The derivative $\partial_{x}\Phi(t,x)$ is given by
$1-\partial_{x}\widehat \Phi(t,x)$ and is also uniformly bounded away from zero for small $\zeta>0$. We also see now that $\cG[\Phi]$ is well-defined.

The rest of the proof is standard. First one shows that for sufficiently small $\zeta,T$ that $\cG$ maps $\cB$ into $\cB$ and is a contraction. Note that the $\eps$-dependent estimates \eqref{est_v1} are crucial for the self-mapping and contraction properties. By the contraction mapping theorem there exists a unique solution $\Phi$ of \eqref{eq:Phi1} on some small time-interval $[0,T]$.

The local-in-time solution of the regularized problem (for any fixed $\eps>0$) is easily extended to $t\in [0,\infty)$ by standard arguments, noting that the a-priori bound
\begin{align}
    \|v_x\|_{-\ga+p} \leq C(\eps)\|\om\|_p 
\end{align}
holds independent of the length of the time interval $[0, T]$.
\end{proof}

In view of the bounds in Proposition \ref{prop4}, we can now complete the proof of Theorem \ref{thm:shorttimeexist} by using the Arzel\'{a}-Ascoli theorem on the sequence of solutions of the regularized problem as $\eps\to 0$. This gives a solution to \eqref{eq_main}, which is $C^1([0, T]\times \R)$.


\subsection{Proof of Theorem \ref{thm_MR} (Singularity Formation)}

\subsubsection{Preliminaries}

We need a few preliminary propositions first.

\begin{proposition}\label{prop1}
Let $a: [0, T)\mapsto \R$ be a smooth function with $a_0 := a(0) > 0$ and define
\begin{align}\label{def_phi1}
    \phi(t, x) = a^p(t) f\left(\frac{x}{a(t)}\right)
\end{align}
with $f(z) = (z+1)^p-1$. Then
\begin{align}\label{eq_repr_u}
    -u[\phi(t, \cdot)] = a(t)^{1-\ga + p}U\left(\frac{x}{a(t)}\right)
\end{align}
where
\begin{align}
    U(z) = \int_0^\infty \left(\frac{1}{|y-z|^\ga}-\frac{1}{|y+z|^\ga}\right) f(y)~dy.
\end{align}
Moreover,
\begin{itemize}
    \item $U'(0) > 0$
    \item $U(x) > 0, \quad (x > 0)$
    \item $U(x) \sim C x^{1-\ga+p}$ as $x\to \infty$ with some $C> 0$.
\end{itemize}
\end{proposition}

\begin{proof}
We have
\begin{align}
    -u[\phi(t, \cdot)](x) &= a^p \int_{0}^\infty \left(\frac{1}{|y-x|^\ga}-\frac{1}{|x+y|^\ga}\right)f\left(\frac{y}{a}\right)~dy\\
    &=a^{1-\ga+p}\int_{0}^\infty \left(\frac{1}{|w-\frac{x}{a}|^\ga}-\frac{1}{|w+\frac{x}{a}|^\ga}\right)f(w)~dw\\
    &= a^{1-\ga+p} U\left(\frac{x}{a}\right).
\end{align}
after substituting $y = a w$ in the integral. Hence the representation \eqref{eq_repr_u} holds. From the form of $U$, we directly see that $U(x) > 0$, since the integrand is $>0$. To see that $U(x) \sim C x^{1-\ga+p}$ we compute
\begin{align}
    U(x) = x^{1-\ga+p}\left\{\int_0^\infty \left(\frac{1}{|z-1|^\ga}-\frac{1}{|z+1|^\ga}\right) \left(\left(z+\frac{1}{x}\right)^p - \frac{1}{x^p}\right)~dz\right\}
\end{align}
and note that the integral in curly brackets converges to a positive constant as $x\to \infty$, as can be seen using the dominated convergence theorem. To see $U'(0)>0$, we write
\begin{align}
    U'(0) = 2 \int_0^\infty|y|^{-\ga} f'(y)~dy
\end{align}
and note $f'(y) > 0$. 
\end{proof}

\begin{lemma} \label{lem1} Let $\om$ satisfy all the assumptions of Theorem \ref{thm_MR}.  Let $\bar T$ be the maximal life-time of the smooth solution $\om$. Then for all $t\in[0,\min\{\bar T, T(a_0)\})$ and all $x\geq 1$,
\begin{align}
    \om(t,x) &> \phi(t,x).
\end{align}
Moreover, there exists a $\delta>0$ so that $\om(t,x) > \phi(t,x)$ for $0\leq t\leq \delta, 0 < x < \infty$.
\end{lemma}
\begin{proof}
To prove the statement referring to $x\geq 1$, we first note that for any $t < \bar T, x \geq 1$, there exists a particle trajectory $t\mapsto X(t)$
\begin{align}
    \dot{X}(t) = u[\om(t, \cdot)](X(t), t), ~~X(0) = X_0
\end{align} 
such that $X(t) = x$. The assumptions of Theorem \ref{thm_MR} imply in particular that $\om(t, x)$ is always non-negative, so that $u[\om]$ is non-positive for $x>0$. Hence the particle trajectory originates from a point $X_0 \geq x\geq 1$ and we have, by \eqref{eq_phi_less_than_x_plus_a_0}, 
\begin{align*}
    \om(t, x) = \om(0, X_0) \geq (1+\eps)\phi(0, X_0) > (X_0+a_0)^p \geq (x+a_0)^p \geq \phi(t, x)
\end{align*}
since our choice of $\eps$ (see \eqref{choice_eps}) guarantees the inequality
\begin{align*}
    \frac{\eps}{1+\eps} > \frac{a_0^p}{(1+a_0)^p}
\end{align*}
implying $(1+\eps)\phi(0, X_0) > (X_0+a_0)^p$ for all $X_0 \geq 1$. 

To argue that the second statement of the Lemma holds, we show first the existence of an $0 <\delta_1$ and an $0 < b < 1$ such that $\om > \phi$ on $0\leq t \leq \delta, 0\leq x \leq b$. The assumption $\om(0, x) > (1+\eps) \phi(0, x)$ implies $\del_x \om(0, x) > \del_x \phi(0, x)$ for $x\in [0, b]$ for some small $b > 0$. Smoothness of $\om$ in time implies that $\del_x \om(t, x) > \del_x \phi(t, x)$ for $(t, x)\in [0, \delta_1]\times[0, b]$ and some $\delta_1 > 0$. Because of $\om(t, 0) = \phi(t, 0) = 0$, we then conclude by integrating with respect to $x$ that $\om(t, x) > \phi(t, x)$ on $[0, \delta_1]\times[0, b]$. To complete the proof of the second part of the proposition, we choose a $\delta_2 > 0$ such that $\om > \phi$ on $[0, \delta_2]\times [b, 1]$ and set $\delta := \min\{\delta_1, \delta_2\}$.
\end{proof}

\begin{lemma}\label{lem2}
Let all the assumptions of Theorem \ref{thm_MR} hold. Define a time $T^{*}$ by
\begin{align*}
    T^{*} = \sup\{ 0 \leq t < \min\{\bar T, T(a_0)\} : \om(\tau, x) > \phi(\tau, x)~\text{for all}~(\tau, x)\in [0, t]\times (0, \infty)\}.
\end{align*}
Suppose also for this proposition that
\begin{align}\label{ass_ineq_phi1}
    \phi_t + u[\phi]\phi_x < 0 \quad (x > 0).
\end{align}
Then if $T^{*} < \bar T$
\begin{align}\label{claim1}
    \del_{x}\om(T^{*}, 0) > \del_x \phi(T^{*}, 0).
\end{align}
As a consequence, there exists a $b > 0$ such that $\om(T^{*}, x)>\phi(T^{*}, x)$  for all $0 < x < b$.
\end{lemma}
\begin{proof}

The supremum defining $T^{*}$ is $>0$ because of Lemma \ref{lem1}. 
The equation \eqref{eq_main} and $u[\om](t, 0) = 0$ imply for short times
\begin{align}\label{eq_om_x_log}
    \frac{d}{dt} \ln \del_x \om(t, 0) = - (\del_x u[\om])(t, 0). 
\end{align}
because $\del_x \om(t, 0) > 0$ for small $t>0$. Observe that $\om(t, \cdot) > \phi(t, \cdot)$ for all $t < T^*$. Using this, we get for $t < T^{*}$
\begin{align}
     - (\del_x u[\om])(t, 0) &= - \lim_{x\to 0^+}\frac{u[\om](t, x)}{x}
     \geq \lim_{x\to 0^+} \frac{1}{x}\int_{0}^\infty K(x, y) \phi(t, y)~dy\\
    &\geq - \lim_{x\to 0^+}\frac{u[\phi](t, x)}{x} = - (\del_x u[\phi])(t, 0).\label{eq_prop_aux1}
\end{align}
Moreover the assumption \eqref{ass_ineq_phi} implies, on account of $\phi_t(t, 0) = 0$ and \eqref{ass_ineq_phi1},
\begin{align*}
\frac{1}{x} \int_{0}^x \del_x \phi_t(t, y)~dy = \frac{1}{x} \phi_{t}(t, x) < - \frac{u[\phi](t, x) \phi_x(t, x)}{x} 
\end{align*}
from which by taking the limit $x\to 0$ and using $u[\phi](t, 0) = 0$ we get 
\begin{align}\label{eq_prop_aux2}
    \del_{t} \phi_x(t, 0) \leq - (\del_x u[\phi])(t, 0) \phi_x(t, 0).
\end{align}
Combining \eqref{eq_om_x_log}, \eqref{eq_prop_aux1} and \eqref{eq_prop_aux2}, we get for small $t>0$
\begin{align}\label{eq_ineq_om_x_phi_x}
    \frac{d}{dt}\ln \om_{x}(t, 0) \geq \frac{d}{dt} \ln \phi_{x}(t, 0) > 0.
\end{align}
The inequality \eqref{eq_ineq_om_x_phi_x} remains valid as long as $\om_{x}(t, 0) > 0$. By direct calculation, one finds that $\frac{d}{dt} \ln \phi_{x}(t, 0) > 0$ for all $t < T(a_0)$ and hence the inequality holds up to $T^{*}$. By taking into account that $\om_x(t, 0) > \phi_x(t, 0)$ for small positive $t>0$ and integrating \eqref{eq_ineq_om_x_phi_x} up to $T^*$ we arrive at \eqref{claim1}.
\end{proof}

\begin{proposition}\label{prop_main}
 Let $a_0, \phi$ be as in Proposition \ref{prop1}. 
Suppose $\om$ is a smooth, odd solution of \eqref{eq_main} and that 
all the assumptions of Theorem \ref{thm_MR} hold. 
 Suppose moreover  for now  that
\begin{align}\label{ass_ineq_phi}
    \phi_t + u[\phi]\phi_x < 0 \quad (x > 0).
\end{align}
Then $\om(t, x) > \phi(t, x)$ for all $x > 0$ and  for times $t < \min\{\bar T, T(a_0)\} $.
\end{proposition}

\begin{proof} 
Define $T^{*}$ as in Lemma \ref{lem2} and assume that the conclusion of the Proposition is false, i.e. $T^{*} < \min\{\bar T, T(a_0)\}$. Then there exists a sequence $(\tau_n, x_n)$ with $\tau_n \to T^{*}$, $\om(\tau_n, x_n) \leq \phi(\tau_n, x_n)$ and by Lemmas \ref{lem1} and \ref{lem2}, $0 < b \leq x_n \leq 1$. By passing to a subsequence, we have $x_n \to x^{*}\in (b, 1]$.
As a consequence, we have $\om(T^{*}, x^*) = \phi(T^{*}, x^*)$ for some $x^* > 0$.

Let $X(t)$ denote any particle trajectory defined by
\begin{align}
    \dot{X}(t) = u(X(t), t), ~~X(0) = X_0
\end{align}
where $X_0 > 0$ and such that $X(T^{*}) = x^*$.  Observe that 
\begin{align*}
    \left.\frac{d}{dt}\om(t, X(t))\right|_{t=T^{*}} \leq \left.\frac{d}{dt} \phi(t, X(t))\right|_{t=T^{*}}
\end{align*}
since otherwise by backtracking the trajectory we see that at all times $T^{*}-\eta$ with small $\eta>0$ and positions $X(T^*-\eta)$,  $\om(T^*-\eta, X(T^*-\eta)) < \phi(T^*-\eta, X(T^*-\eta))$ holds, in contradiction to the definition of $T^*$.  
Then,
\begin{align}
    &0 = \left.\frac{d}{dt}\om(t, X(t))\right|_{t=T^{*}} \leq \left.\frac{d}{dt} \phi(t, X(t))\right|_{t=T^{*}} = \left.(\phi_t + u[\om]\phi_x) \right|_{t=T^{*}}\\
    &\leq (\phi_t + u[\phi] \phi_x)(T^{*}, x^*) < 0 
\end{align}
where we have used $\om(T^{*}, x) \geq \phi(T^{*}, x)$ for all $x\geq 0$ to conclude
\begin{align}
    u[\om]\leq u[\phi].
\end{align}
Hence in summary we get at $(T^{*}, x^*)$ the relationship
\begin{align}
    0 = (\phi_t + u[\phi] \phi_x)(T^{*}, x^*) < 0, 
\end{align}
a contradiction.
\end{proof}

\begin{proposition}\label{prop2}  Let $p = \frac{1}{2}\gamma$.
There exists a positive constant $c > 0$ such that
\begin{align}
    0 < c \leq \frac{U(z) f'(z)}{-p f(z) + z f'(z)}\quad (z > 0)
\end{align}
\end{proposition}

\begin{proof}
We calculate $-p f(z) + z f'(z) = -p ((z+1)^p - 1) + z p (z+1)^{p-1}$ and note that 
$-p f(z) + z f'(z) > 0$ for all $z > 0$. Now observe that by Proposition \ref{prop1}, 
$U(z)\sim c_1 z$ for  small $z > 0$ with  some $c_1>0$ and furthermore that $-p f(z) + z f'(z)\sim p(1-p) z$ for small $z$. Hence there exists a $z_1 > 0$ such that
\begin{align}
c_2 \leq  \frac{U(z) f'(z)}{-p f(z) + z f'(z)} \quad (0 < z \leq z_1)
\end{align}
for some $c_2 >0$. As $z\to \infty$, $U(z)\sim C z^{1-\ga+p}, f'(z)\sim z^{p-1}$ and $-p f(z) + z f'(z) \to p$. Using again  $p = \half \ga$,  we conclude the existence of  an $z_2>z_1$ so that
\begin{align}
c_3 \leq  \frac{U(z) f'(z)}{-p f(z) + z f'(z)} \quad ( z_2 \leq z).
\end{align}
 for some $c_3 > 0$. 
The statement of the Proposition now follows since $\frac{U(z) f'(z)}{-p f(z) + z f'(z)}$ is continuous in $z$ and never zero in $[z_1, z_2]$.
\end{proof}

\subsubsection{Proof of Theorem \ref{thm_MR}}

We now turn to Theorem \ref{thm_MR}. We need to check the following:
$\phi(t, x)$ satisfies
    \begin{align}\label{cond_1}
       \phi_t + u[\phi]\phi_x < 0\quad (x>0).
     \end{align}
To prove this, we first compute the left hand side using \eqref{def_phi} and Proposition \ref{prop1}:
\begin{align}
    \phi_t + u[\phi]\phi_x = \left[\dot{a} a^{p-1} (p f(z) - z f'(z)) - a^{2p-\ga} U(z) f'(z)\right]_{z=\frac{x}{a}}
\end{align}
and so $\phi_t + u[\phi]\phi_x < 0$ is equivalent to
\begin{align}\label{eq_aux}
    (-\dot{a}) < a^{1-\ga+p} \frac{ U(z) f'(z)}{-p f(z) + z f'(z)}
\end{align}
for all $z > 0$. By Proposition \ref{prop2}, \eqref{eq_aux} is implied by 
\begin{align}
(-\dot{a}) < c a^{1-\ga+p} 
\end{align}
so $\dot{a} = - \half c a^{1-p}$ is sufficient for \eqref{cond_1} to hold,  since $p =\half \ga$. By applying Proposition \ref{prop_main}, we see that $\om > \phi$ as long as $t < \min\{\bar T, T(a_0)\}$. Note that now necessarily $\bar T\leq T(a_0)$, since in the case $T(a_0) < \bar T$ the inequality $\om(T(a_0), x) \geq \phi(T(a_0), x) = x^p$ and $\om(t, 0) = \phi(t, 0)$ would imply that $\om$ has an infinite slope at $t = T(a_0), x=0$. This completes the proof of Theorem \ref{thm_MR}.

\section{Acknowledgements}
The authors would like to thank the anonymous reviewer for a careful reading of this manuscript, many helpful comments, and in particular for pointing out a gap in the first version. We would also like to thank D. Li for very helpful comments on a preprint version of this paper, indicating a related gap. Vu Hoang wishes to thank the National Science Foundation for support under grants NSF DMS-1614797 and NSF DMS-1810687.

%
%
%

\end{document}